\documentclass[10pt, a4paper]{amsart} 
\usepackage[left=2.50cm, right=2.50cm, top=2.50cm, bottom=2.50cm]{geometry} 
\usepackage{amscd, amsfonts, amsmath, amssymb, amsthm, arydshln, fixmath, graphicx, mathrsfs, overpic, tikz, tikz-cd}
\usepackage{hyperref}
\usepackage[all]{xy} 
\usepackage{tikz}
\usepackage{tikz}
\usetikzlibrary{arrows}
\usetikzlibrary{decorations.pathmorphing}
\usepackage{dsfont}
\makeindex

\theoremstyle{definition} 
\newtheorem{Unity}{Unity}[section] 
\newtheorem*{Definition*}{Definition} 
\newtheorem{Definition}[Unity]{Definition} 

\theoremstyle{plain} 
\newtheorem*{Theorem*}{Theorem}
\newtheorem{Theorem}[Unity]{Theorem}
\newtheorem{Proposition}[Unity]{Proposition}
\newtheorem{Corollary}[Unity]{Corollary}
\newtheorem{Lemma}[Unity]{Lemma}

\theoremstyle{remark} 
\newtheorem*{Remark*}{Remark}

\numberwithin{Unity}{section}

\newcommand{\GL}{\mathrm{GL}}
\newcommand{\Ima}{\mathrm{Im}}

\newcommand{\Hom}{\mathrm{Hom}}

\newcommand{\peri}{\mathrm{per}}

\newcommand{\chara}{\mathrm{char\,}}
\newcommand{\orb}{\mathrm{orb}}
\newcommand{\ord}{\mathrm{ord}}
\newcommand{\OBT}{\mathrm{OBT}}

\begin{document}

\title{On the number of Frobenius periodic vector bundles on elliptic curves}
\author{Lingguang Li}
\address{School of Mathematical Sciences,
Key Laboratory of Intelligent Computing and Applications (Tongji University), Ministry of Education, Shanghai 200092, CHINA}
\email{LiLg@tongji.edu.cn}
\author{Niantao Tian}
\address{School of Mathematical Sciences,
Key Laboratory of Intelligent Computing and Applications (Tongji University), Ministry of Education, Shanghai 200092, CHINA}
\email{tianniantao@tongji.edu.cn}
\begin{abstract} 
This paper counts Frobenius-periodic vector bundles on elliptic curves over an algebraically closed field of characteristic $p>0$. By translating the problem into continuous representations of the étale fundamental group, it derives explicit generating functions and exact-period formulas, with separate treatments of the ordinary and supersingular cases.
\end{abstract}
\maketitle

\section{Introduction}
Let $k$ be an algebraically closed field of characteristic $p>0$, $X/k$ an elliptic curve with absolute Frobenius morphism $F_X:X\rightarrow X$. A vector bundle $E$ on $X$ is called Frobenius-periodic if $(F_X^n)^*E\simeq E$ for some $n>0$. In this paper we study Frobenius-periodic vector bundles not through the geometry of the whole moduli space, but through finite monodromy representations of the \'etale fundamental group. More precisely, for fixed $n$ and $r$, we count the isomorphism classes of rank $r$ vector bundles satisfying $(F_X^n)^*E\simeq E$, and then extract the number of those whose minimal Frobenius period is exactly $n$.

The starting point is the theorem of Lange and Stuhler \cite{LaSt77}, which says that a vector bundle fixed by an iterate of Frobenius is trivialized by a finite étale cover. Thus Frobenius-periodic vector bundles are naturally related to continuous representations of the étale fundamental group $\pi_1^{\acute{e}t}(X,x)$. Under this correspondence, if $E_\rho$ corresponds to a representation $\rho:\pi_1^{\acute{e}t}(X,x)\to \GL_r(k)$, then $F_X^*E_\rho$ corresponds to the representation $\rho^{(p)}$ obtained by applying $a\mapsto a^p$ to the matrix entries. Consequently, the condition $(F_X^n)^*E_\rho\simeq E_\rho$ becomes the condition $\rho^{(p^n)}\simeq \rho$.

A different viewpoint was developed by Ducrohet and Mehta.  Using Hrushovski's theorem on moduli space of semistable vector bundles, they proved that Frobenius-periodic points are Zariski dense in the moduli space of semistable vector bundles over a smooth projective curve of genus at least two \cite{DuMe10}.  Their result shows that Frobenius-periodic bundles are abundant from the qualitative point of view of moduli theory.  In contrast, the present paper turns from qualitative density to explicit enumeration: in the case of elliptic curves, we count Frobenius-periodic bundles rank by rank and period by period.

One basic observation used in the paper is that the condition $\rho^{(p^n)}\simeq \rho$ is equivalent to saying that $\rho$ is isomorphic over $k$ to a representation defined over the finite field $\mathbb F_{p^n}\subsetneq k$. This follows from the Lang--Steinberg theorem applied to $\GL_r(k)$. Therefore, the enumeration of Frobenius-periodic bundles becomes a finite representation-theoretic counting problem.

Our main result gives an explicit enumeration of Frobenius-periodic vector bundles on elliptic curves. 

\begin{Theorem}[Proposition~\ref{ordinary}, Proposition~\ref{supersingular}]
    Let $k$ be an algebraically closed field of characteristic $p>0$, $X$ an elliptic curve over $k$, $n\in\mathbb{N}_{>0}$, $N_r^{= n}$ denote the number of isomorphism classes of rank $r$ vector bundles $E$ on $X$ whose Frobenius period is exactly $n$, $\mu$ denote the M\"obius function. Then we have 
    $$N_r^{= n}=\begin{cases}
    \sum_{t\mid n}\mu\!\left(\frac{n}{t}\right)\sum\limits_{\substack{z_1,\dots,z_r\geq 0\\z_1+2z_2+\cdots+rz_r=r}}\prod_{s=1}^r\binom{z_s-1+\sum_{d\mid s}\frac{1}{d}\sum_{e\mid d}\mu\!\left(\frac{d}{e}\right)(p^{te}-1)^2}{z_s} ,&\text{ if } X \text{ is ordinary}\\
    \sum_{t\mid n}\mu\!\left(\frac{n}{t}\right)\sum\limits_{\substack{z_1,\dots,z_r\geq 0\\z_1+2z_2+\cdots+rz_r=r}}\prod_{d=1}^r\binom{z_d-1+\frac{1}{d}\sum_{e\mid d}\mu\!\left(\frac{d}{e}\right)(p^{te}-1)^2}{z_d} ,&\text{ if } X \text{ is supersingular}
    \end{cases}$$
\end{Theorem}
This viewpoint is closely connected with the classical study of torsion points on abelian varieties. Recall that if \(A\) is an abelian variety of dimension \(g\) over an algebraically closed field of characteristic \(p>0\), then for every integer \(m\) prime to \(p\), one has $A[m]\cong (\mathbb Z/m\mathbb Z)^{2g}$ \cite{MilneAV}. For the elliptic curve \(X\), \(\operatorname{Pic}^0(X)\) is the Jacobian variety of \(X\), and the rank-one case of our problem recovers exactly a torsion-point counting problem on \(\operatorname{Pic}^0(X)\): $(F_X^n)^*L\simeq L\Leftrightarrow L^{\otimes(p^n-1)}\simeq \mathcal O_X$, so Frobenius-periodic line bundles are precisely \((p^n-1)\)-torsion points. Thus our formula gives \(N_1^{\mid n}=\sum_{t\mid n}N_1^{=t}=(p^n-1)^2\) according to Corollary~\ref{Cor}, which coincides with the order of $\operatorname{Pic}^0(X)[p^n-1]$. 

\section{Preliminaries}

We introduce some basic definitions and propositions which will be applied to the construction of generating function.

\begin{Definition}
Let $k$ be an algebraically closed field of characteristic $p>0$, $X/k$ an elliptic curve. We say that $X$ is \emph{ordinary} if $X[p](k)\cong \mathbb Z/p\mathbb Z$. We say that $X$ is \emph{supersingular} if $X$ is not ordinary.
\end{Definition}

\begin{Lemma}[{\cite[Proposition~5.13, Proposition~5.14]{Kun17}}]
    Let $k$ be an algebraically closed field of characteristic $p>0$, $X/k$ an elliptic curve, $x\in X(k)$. If $X$ is ordinary, then $\pi^{\acute{e}t}_1(X,x)\cong \left(\prod_{\ell\neq p}\mathbb Z_\ell^2\right)\times \mathbb Z_p$; if $X$ is supersingular, then $\pi^{\acute{e}t}_1(X,x)\cong \prod_{\ell\neq p}\mathbb Z_\ell^2$. Here $\mathbb{Z}_q:=\varprojlim \mathbb{Z}/q^n\mathbb{Z}$ for any prime number $q$.
\end{Lemma}

\begin{Lemma}[{\cite[Proposition~1.0.2]{LaSt77}}]\label{LaStequivalence}
    Let $k$ be an algebraically closed field, $X$ a scheme proper over $k$, $x\in X(k)$. There exists a natural bijection between isomorphism classes of continuous representations of $\pi_1^{\acute{e}t}(X,x)\rightarrow \GL_r(k)$ and isomorphism classes of \'etale trivializable vector bundles on $X$ of rank $r$.
\end{Lemma}

\begin{Lemma}[{\cite[Section~1.1]{LaSt77}}]\label{Frobenius}
    Let $k$ be an algebraically closed field of characteristic $p>0$, $X$ a scheme proper over $k$, $F_X:X\rightarrow X$ the absolute Frobenius morphism, $x\in X(k)$. If $E_{\rho}$ is an \'etale trivializable vector bundle on $X$ corresponding to a representation $\rho:\pi_1^{\acute{e}t}(X,x)\rightarrow \GL_r(k)$, then $F_X^*E_{\rho}$ corresponds to the $p$-th power representation $\rho^{(p)}:\pi_1^{\acute{e}t}(X,x)\rightarrow \GL_r(k),g\mapsto (a_{ij}^{p})$ where $g\in \pi_1^{\acute{e}t}(X,x)$ and $\rho(g)=(a_{ij})$.
\end{Lemma}

\begin{Lemma}[{\cite[Proposition~2.3.5]{Wil24}}]\label{determinedbygenerator}
    Let $G$ be a topological group and $S$ a topological generating set for $G$. Let $f_1, f_2:G\rightarrow H$ be continuous homomorphisms to a Hausdorff topological group $H$. If $f_1$ and $f_2$ agree on S then $f_1=f_2$.
\end{Lemma}

\begin{Lemma}[{\cite[Example~5.3]{CLV20}}]
    $\hat{\mathbb{Z}}:=\varprojlim\limits_{n\in \mathbb{N}}\mathbb{Z}/n\mathbb{Z}\cong\prod_{\ell \text{ prime}}\mathbb Z_\ell$ is a topologically finitely generated group with generator $\mathbf{1}$, where $\mathbb{Z}_{\ell}:=\varprojlim\limits_{m\in \mathbb{N}}\mathbb{Z}/l^m\mathbb{Z}$ is the $\ell$-adic integer ring.
\end{Lemma}

\begin{Lemma}
    Let $k$ be a field of characteristic $p>0$. Then we have $\Hom_{\mathrm{cont}}(\hat{\mathbb{Z}},k^*)=\Hom_{\mathrm{cont}}(\prod_{\ell\neq p}\mathbb Z_\ell,k^*)$ where $k^*$ is endowed with the discrete topology.
\end{Lemma}

\begin{proof}
    Note that $\hat{\mathbb Z}\cong \mathbb{Z}_p\times \prod_{\ell\neq p}\mathbb Z_\ell$. Take $\chi\in \Hom_{\mathrm{cont}}(\hat{\mathbb{Z}},k^*)$ and restrict it to $\mathbb{Z}_p$, we have $\Ima(\chi|_{\mathbb{Z}_p})$ is a finite $p$-group. Since $\chara k=p$, if the multiplicative order of $x\in \Ima(\chi|_{\mathbb{Z}_p})$ is divisible by $p$, we obtain $(x-1)^{p^n}=0$ for some $n\geq 0$, which implies that $x=1$. So $\Ima(\chi|_{\mathbb{Z}_p})=\{1\}$, which implies that $\Hom_{\mathrm{cont}}(\hat{\mathbb{Z}},k^*)=\Hom_{\mathrm{cont}}(\prod_{\ell\neq p}\mathbb Z_\ell,k^*)$.
\end{proof}

\begin{Lemma}\label{Numberofcharacter}
    Let $p$ be a prime integer. Then $|\Hom_{\mathrm{cont}}(\prod_{\ell\neq p}\mathbb Z_\ell,\mathbb{F}_{p^n}^*)|=p^n-1$ where $\mathbb{F}_{p^n}^*$ is endowed with the discrete topology.
\end{Lemma}

\begin{proof}
    For any $x\in \mathbb{F}_{p^n}^*$, we have $p\nmid \ord(x)$. Set $m=\ord(x)$ and define $\chi_x:\mathbb{Z}\rightarrow \mathbb{F}_{p^n}^*, i\mapsto x^i$. It follows that $\ker(\chi_x)=m\mathbb{Z}$, and we have $\mathbb{Z}\twoheadrightarrow \mathbb{Z}/m\mathbb{Z}\hookrightarrow \mathbb{F}_{p^n}^*,i\mapsto \bar{i}\mapsto x^i$. We have a natural projection $\hat{\mathbb{Z}}\twoheadrightarrow \mathbb{Z}/m\mathbb{Z}$, composing with $\mathbb{Z}/m\mathbb{Z}\hookrightarrow \mathbb{F}_{p^n}^*$, we obtain 
    $$\chi_x\in \Hom_{\mathrm{cont}}(\hat{\mathbb{Z}},k^*):\hat{\mathbb{Z}}\rightarrow \mathbb{Z}/m\mathbb{Z}\rightarrow \mathbb{F}_{p^n}^*, \mathbf{1}\mapsto \bar{1}\mapsto x.$$
Consequently, every $x\in \mathbb{F}_{p^n}^*$ determines a unique $\chi_x\in \Hom_{\mathrm{cont}}(\hat{\mathbb{Z}},k^*)$ such that $\mathbf{1}\mapsto x$ by Lemma~\ref{determinedbygenerator}. Hence $|\Hom_{\mathrm{cont}}(\prod_{\ell\neq p}\mathbb Z_\ell,\mathbb{F}_{p^n}^*)|=|\Hom_{\mathrm{cont}}(\hat{\mathbb{Z}},k^*)|=p^n-1$. 
\end{proof}

\begin{Lemma}[Lang--Steinberg {\cite[Theorem 10.1]{Ste68}}]\label{Lang}
Let $G$ be a connected linear algebraic group over an algebraically closed field $k$, and $\sigma:G\rightarrow G$ an endomorphism such that the fixed point group $G^\sigma$ is finite. Then the morphism $L_\sigma:G\rightarrow G$, $g\mapsto g^{-1}\sigma(g)$ is surjective.
\end{Lemma}

\begin{Lemma}[{\cite[Proposition~6.1.1]{Web16}}]\label{6.1.1}
    Let $k$ be a field of characteristic $p$, $ G=\langle g\mid g^{p^n}=1\rangle $ the cyclic group of order $p^n$, $kG$ the group algebra. Then there is an isomorphism of $k$-algebras $ kG\simeq k[X]/(X^{p^n})$.
\end{Lemma}

\begin{Lemma}[{\cite[Proposition~6.1.2]{Web16}}]\label{6.1.2}
        Let $k$ be a field of characteristic $p$. Every finitely generated $k[X]/(X^{p^n})$-module is a direct sum of cyclic modules $ U_r=k[X]/(X^r)$, $1\leq r\leq p^n$. The only simple module among these is $U_1$. Moreover, each $U_r$ has a unique composition series and is therefore indecomposable. Consequently, if $G$ is cyclic of order $p^n$, then the group algebra $kG$ has precisely $p^n$ indecomposable modules up to isomorphism, namely one of each dimension $r$, where $ 1\leq r\leq p^n. $
\end{Lemma}

\section{The number of Frobenius periodic vector bundles on elliptic curves}

\begin{Proposition}\label{Frobenius-periodic-etale-representation}
Let $k$ be an algebraically closed field of characteristic $p>0$, $n\in \mathbb{N}_{> 0}$, $X$ a reduced connected scheme proper over $k$, $x\in X(k)$ and $F_X:X\rightarrow X$ the absolute Frobenius morphism. For a continuous representation $\rho:\pi^{\acute{e}t}_1(X,x)\rightarrow \GL_r(k)$, we denote its associated vector bundle by $E_{\rho}$. Then the following conditions are equivalent:
\begin{enumerate}
    \item $F_X^{n*}E_{\rho}\cong E_{\rho}$.
    \item There exists a continuous representation $\rho_n:\pi^{\acute{e}t}_1(X,x)\rightarrow \GL_r(\mathbb{F}_{p^n})\hookrightarrow \GL_r(k)$ such that $\rho\simeq \rho_n.$
\end{enumerate}
\end{Proposition}

\begin{proof}
Define the map $\sigma_n:\GL_r(k)\rightarrow \GL_r(k), (a_{ij})\mapsto (a_{ij}^{p^n})$ and we denote $\sigma_n(A)$ by $A^{(p^n)}$ for any $A\in \GL_r(k)$. Since $\chara(k)=p$, this map is a homomorphism. Since the polynomial $x^{p^n}=x$ has $p^n$ roots at most, it follows that the fixed point group $\GL_r(k)^{\sigma_n}\subseteq \GL_r(\mathbb F_{p^n})$ is finite.

$(1)\Rightarrow (2)$ Suppose $F_X^{n*}E_{\rho}\cong E_{\rho}$. By Lemma~\ref{LaStequivalence} and Lemma~\ref{Frobenius}, we have $\rho^{(p^n)}\simeq\rho$. Hence there exists $A\in \GL_r(k)$ such that $\rho^{(p^n)}(\gamma)=A^{-1}\rho(\gamma)A$ for every $\gamma\in \pi_1^{\acute{e}t}(X,x)$.

By Lemma~\ref{Lang}, for the algebraic group $\GL_r(k)$ over $k$, the morphism $\GL_r(k)\rightarrow \GL_r(k)$, $B\mapsto  B^{-1}B^{(p^n)}$ is surjective. Thus we may choose $B\in \GL_r(k)$ such that $A= B^{-1}B^{(p^n)}$. Define the representation $\rho_n:\pi_1^{\acute{e}t}(X,x)\rightarrow \GL_r(k),\gamma\mapsto B\rho(\gamma)B^{-1}$ for any $\gamma\in \pi_1^{\acute{e}t}(X,x)$. Then $\rho_n\simeq \rho$ and 
$$\begin{aligned}
    \rho_n(\gamma)^{(p^n)} &=B^{(p^n)}\rho(\gamma)^{(p^n)}(B^{-1})^{(p^n)}\\
    &=B^{(p^n)}\rho^{(p^n)}(\gamma)(B^{(p^n)})^{-1}\\
    &=B^{(p^n)}A^{-1}\rho(\gamma)A(B^{(p^n)})^{-1}\\
    &=B^{(p^n)}((B^{(p^n)})^{-1} B)\rho(\gamma) (B^{-1}B^{(p^n)})(B^{(p^n)})^{-1}\\
    &=B\rho(\gamma)B^{-1}\\
    &=\rho_n(\gamma)
\end{aligned}$$
for any $\gamma\in \pi_1^{\acute{e}t}(X,x)$. Therefore every matrix $\rho_n(\gamma)$ is fixed by $\sigma_n$, which implies that $\Ima(\rho_n)\subseteq\GL_r(k)^{\sigma_n}\subseteq \GL_r(\mathbb F_{p^n})$. Hence $\rho$ is similar over $k$ to a representation defined over $\mathbb F_{p^n}$.

$(2)\Rightarrow (1)$ Assume that $\rho$ is similar over $k$ to an $\mathbb F_{p^n}$-valued representation $\rho_n$. Thus there exists $B\in\GL_r(k)$ such that $\rho_n(\gamma):=B^{-1}\rho(\gamma)B$ lies in $\GL_r(\mathbb F_{p^n})$ for every $\gamma\in\pi_1^{\acute{e}t}(X,x)$. Since every element of $\mathbb F_{p^n}$ is fixed by $\sigma_n$, we have $\rho_n^{(p^n)}=\rho_n$. It follows that $\rho^{(p^n)}$ is isomorphic to $\rho$, which implies that $(F_X^n)^*E_\rho\simeq E_\rho$ by Lemma~\ref{LaStequivalence} and Lemma~\ref{Frobenius}.
\end{proof}

\begin{Proposition}\label{directproduct}
Let $k$ be an algebraically closed field of characteristic $p>0$, $X$ an ordinary elliptic curve over $k$, $x\in X(k)$, $\rho:\pi^{\acute{e}t}_{1}(X,x)\rightarrow \GL_r(k)$ a continuous $k$-representation, where $\GL_r(k)$ is endowed with the discrete topology. Then the image of $\rho$ is naturally an internal direct product $\rho(\pi^{\acute{e}t}_{1}(X,x))=A\times C$, where $A$ is a finite abelian group of order prime to $p$, and $C$ is a finite cyclic $p$-group.
\end{Proposition}

\begin{proof}
Since $X$ is ordinary, we have $\pi^{\acute{e}t}_1(X,x)\cong \left(\prod_{\ell\neq p}\mathbb Z_\ell^2\right)\times \mathbb Z_p$. Set $A:=\rho((\prod_{\ell\neq p}\mathbb Z_\ell^2)\times \{1\})$, $C:=\rho(\{1\}\times \mathbb Z_p)$. Since $(\gamma,z)=(\gamma,1)(1,z)$ for every $(\gamma,z)\in (\prod_{\ell\neq p}\mathbb Z_\ell^2)\times \mathbb Z_p$, we have $\rho(\gamma,z)=\rho(\gamma,1)\rho(1,z)$. Thus $\rho(\pi^{\acute{e}t}_{1}(X,x))=AC$. 

The kernel of $\rho$ is open, so $\rho$ factors through a finite quotient of the profinite group $\pi^{\acute{e}t}_{1}(X,x)$. The group $A$ and $C$ are finite quotients of $\pi^{\acute{e}t}_{1}(X,x)_{p'}:=\prod_{\ell\neq p}\mathbb Z_\ell^2$ and $\mathbb{Z}_p$ respectively. Therefore $A$ is finite abelian and its order is prime to $p$ and $C\cong \mathbb Z/p^a\mathbb Z$ for some $a\geq 0$. Moreover, the two factors $\pi^{\acute{e}t}_{1}(X,x)_{p'}\times\{1\}$ and $\{1\}\times\mathbb Z_p$ commute inside $\pi^{\acute{e}t}_{1}(X,x)$. Hence their images $A$ and $C$ commute inside $\GL_r(k)$. Since $|A|$ is prime to $p$ and $|C|$ is a power of $p$, we have $A\cap C=\{1\}$. It follows that $\rho(\pi^{\acute{e}t}_{1}(X,x))=AC\cong A\times C$.
\end{proof}

\begin{Proposition}
    Let $k$ be an algebraically closed field of characteristic $p>0$, $X$ an ordinary elliptic curve over $k$, $x\in X(k)$. Then there exists a bijection between the isomorphism classes of indecomposable $k$-representations of $\pi_1^{\acute{e}t}(X,x)$ and the pairs $(\chi,[\rho_m])$ where $\chi\in \Hom_{\mathrm{cont}}(\prod_{\ell\neq p}\mathbb Z_\ell^2,k^*)$ and $\rho_m$ is an indecomposable representation of $\mathbb{Z}_p$.
\end{Proposition}

\begin{proof}
    Since $X$ is an ordinary elliptic curve, its
étale fundamental group is
\[
\pi_1^{\acute{e}t}(X,x)\simeq
\left(\prod_{\ell\neq p}\mathbb Z_\ell^2\right)\times \mathbb Z_p.
\]
Here $\ell$ runs over prime numbers different from $p$. Thus $\prod_{\ell\neq p}\mathbb Z_\ell^2$ is the prime-to-$p$ part, while $\mathbb Z_p$ is the pro-$p$ part. 

We now describe the indecomposable continuous $k$-representations of $\pi_1^{\acute{e}t}(X,x)$. By Lemma~\ref{directproduct}, the image of continuous $k$-representations of $\pi_1^{\acute{e}t}(X,x)$ has the form $A\times C_{p^a}$, where $A$ is a finite abelian group of order prime to $p$, and $C_{p^a}\cong \mathbb Z/p^a\mathbb Z$. Thus it is enough to describe $k$-representations of $A\times C_{p^a}$.

Let $(V,\rho)$ be a finite dimensional $k$-representation of $A\times C_{p^a}$. By restriction along $i_A:A\hookrightarrow A\times C_{p^a}$, $a\mapsto (a,1)$, $(V,\rho_A:=\rho\circ i_A)$ is an $A$-representation. Similarly $(V,\rho_C:=\rho\circ i_C)$ is a $C_{p^a}$-representation. 

For any $a\in A$, we have $\rho_A(a)^{\ord(a)}=I$. Since $(p, |A|)=1$, we have $p\nmid \ord(a)$. It follows that $x^{\ord(a)}-1$ is a separable polynomial in $k[x]$, which implies that $\rho_A(a)$ is diagonalizable. Moreover, $A$ is abelian, so for any $a,b\in A$, we have $\rho_A(a)\rho_A(b)=\rho_A(b)\rho_A(a)$. Hence there exists a basis $\{v_1,\cdots, v_n\}$ of $V$ such that $\rho_A(a)v_i=\chi_i(a)v_i$ for any $a\in A$ and any $v_i$ where $\chi_i(a)\in k^*$. Note that for any $a,b\in A$, we have $\chi_i(a)\chi_i(b)v_i=\rho_A(a)\rho_A(b)v_i=\rho_A(ab)v_i=\chi_i(ab)v_i$, which implies that $\chi_i:A\rightarrow k^*$ is in fact a character of $A$. Therefore, we obtain a decomposition of $V$ determined by the characters of $A$: $V=\oplus_\chi V_\chi$, where $V_\chi=\{v\in V: \rho_A(a) v=\chi(a)v\ \text{for all }a\in A\}$ where $\chi\in \Hom_{\mathrm{cont}}(A,k^*)$. 

Fix $\chi\in \Hom_{\mathrm{cont}}(A,k^*)$, note that for any $a\in A$, $c\in C_{p^a}$ and $v\in V_\chi$, we have  $\rho_A(a)\rho_C(c)v=\rho_C(c)\rho_A(a)v=\rho_C(c)\chi(a)v=\chi(a)\rho_C(c)v$. It follows that $\rho_C(c)v\in V_\chi$, and hence $(V_\chi,\rho|_{V_\chi})$ is an $A\times C_{p^a}$-representation. Therefore, the decomposition $V=\oplus_\chi V_\chi$ is actually a decomposition as an $A\times C_{p^a}$-representation which coincides with $(V,\rho)$. 

Then we describe the indecomposable representations of $C_{p^a}$. Let $C_{p^a}=\langle g\rangle$. By Lemma~\ref{6.1.1} and Lemma~\ref{6.1.2}, the indecomposable finite dimensional modules over $k[X]/(X^{p^a})$ are $k[X]/(X^m)$, with $1\leq m\leq p^a$. In matrix language, this means that there exists an indecomposable $k$-representation $\rho_m:C_{p^a}\rightarrow \GL_m(k)$ sending the generator $g$ to the unipotent Jordan block $J_m(1)$. It follows that for any integer $1\leq m\leq p^a$, there exists a unique isomorphism class of indecomposable $k$-representation of $C_{p^a}$ of dimension $m$ and there is no indecomposable $k$-representation of $C_{p^a}$ of dimension greater than $p^a$.

Let $(V',\rho')$ be an indecomposable $k$-representation of $\pi_1^{\acute{e}t}(X,x)$, then the image of $\rho'$ is of the form $A\times C_{p^a}$, where $A$ is a finite abelian group of order prime to $p$ and $C_{p^a}\cong \mathbb Z/p^a\mathbb Z$ for some $a\geq0$. Then $(V',\rho')$ is an indecomposable $k$-representation of $A\times C_{p^a}$. Hence there exists a unique character $\chi\in \Hom_{\mathrm{cont}}(A,k^*)$ with $V'=V'_{\chi}$ and $(V',\rho'_C)$ is an indecomposable $k$-representation of $C_{p^a}$. Composing with $\prod_{\ell\neq p}\mathbb Z_\ell^2\twoheadrightarrow A$ and $\mathbb Z_p\twoheadrightarrow C_{p^a}$, we obtain $\chi'\in\Hom_{\mathrm{cont}}(\prod_{\ell\neq p}\mathbb Z_\ell^2,k^*)$ and an indecomposable $k$-representation of $\mathbb Z_p$. Conversely, given a pair $(\chi,\rho_m)$ where $\chi\in\Hom_{\mathrm{cont}}(\prod_{\ell\neq p}\mathbb Z_\ell^2,k^*)$ and $\rho_m:\mathbb Z_p\rightarrow \GL_m(k)$ is an indecomposable $k$-representation, we can define an indecomposable $k$-representation $\rho:\pi_1^{\acute{e}t}(X,x)\rightarrow\GL_m(k),(a,c)\mapsto\rho_{m}(c)\chi(a)=\chi(a)\rho_{m}(c)$ for any $a\in \prod_{\ell\neq p}\mathbb Z_\ell^2$, $c\in \mathbb Z_p$. 

Moreover, two indecomposable $k$-representations $(V_1,\rho_1)$ and $(V_2,\rho_2)$ of $\pi_1^{\acute{e}t}(X,x)$ are isomorphic if and only if there exists an invertible matrix $P$ which induces an isomorphism $V_1\cong V_2$ such that $\rho_1(a,c)=P^{-1}\rho_2(a,c)P$ for any $(a,c)\in \pi_1^{\acute{e}t}(X,x)$. This implies that the corresponding pairs $(\chi_1,\rho_{m_1})$ and $(\chi_2,\rho_{m_2})$ satisfy $\chi_1=\chi_2$ and $\rho_{m_1}(c)\simeq \rho_{m_2}(c)$ for any $c\in\mathbb Z_p$. Hence there is a bijection between the isomorphism classes of the indecomposable $k$-representations of $\pi_1^{\acute{e}t}(X,x)$ and the pairs $(\chi,[\rho_m])$ where $\chi\in \Hom_{\mathrm{cont}}(\prod_{\ell\neq p}\mathbb Z_\ell^2,k^*)$ and $\rho_m$ is an indecomposable representation of $\mathbb{Z}_p$.
\end{proof}

\begin{Proposition}\label{ordinary}
Let $k$ be an algebraically closed field of characteristic $p>0$, $X$ an ordinary elliptic curve over $k$, $x\in X(k)$, $n\in\mathbb{N}_{>0}$, $N_r^{\mid n}(\text{resp. }N_r^{=n})$ denote the number of isomorphism classes of rank $r$ vector bundles $E$ on $X$ satisfying $(F_X^n)^*E\cong E($resp. whose minimal Frobenius period is exactly $n)$. Then we have the following generating series
\[
    \sum_{r\geq 0} N_r^{\mid n}T^r=\prod_{m\geq 1}\prod_{d\geq 1}\left(1-T^{md}\right)^{-\frac{1}{d}\sum_{e\mid d}\mu\!\left(\frac{d}{e}\right)(p^{en}-1)^2}.
\]
Here $\mu$ is the classical M\"obius function. Consequently, we have 
$$N_r^{\mid n}=\sum_{\begin{smallmatrix}
    z_1,\dots,z_r\geq0\\
    z_1+2z_2+\cdots+rz_r=r
\end{smallmatrix} }\prod_{s=1}^r\binom{z_s-1+\sum_{d\mid s}\frac{1}{d}\sum_{e\mid d}\mu\!\left(\frac{d}{e}\right)(p^{ne}-1)^2}{z_s},$$
$$N_{r}^{=n}=\sum_{t\mid n}\mu\!\left(\frac{n}{t}\right)\sum_{\substack{z_1,\dots,z_r\geq 0\\z_1+2z_2+\cdots+rz_r=r}}\prod_{s=1}^r\binom{z_s-1+\sum_{d\mid s}\frac{1}{d}\sum_{e\mid d}\mu\!\left(\frac{d}{e}\right)(p^{te}-1)^2}{z_s}.$$
\end{Proposition}

\begin{proof}
Let $\rho_m:\pi_1^{\acute{e}t}(X,x)\rightarrow \GL_m(k)$ be an indecomposable $k$-representations of $\pi_1^{\acute{e}t}(X,x)$ and denote its corresponding pair by $(\chi,[\rho_m])$. For any $a\in \prod_{\ell\neq p}\mathbb Z_\ell^2$ and $c\in \mathbb{Z}_p$, we have $(\chi(a)\rho_m(c))^{(p^n)}=\chi(a)^{(p^n)}\rho_m(c)^{(p^n)}=\chi^{(p^n)}(a)\rho_m^{(p^n)}(c)\simeq\chi^{(p^n)}(a)\rho_m(c)$. Hence $\rho_m\simeq \rho_m^{(p^n)}$ if and only if $\chi^{(p^n)}=\chi$. 

Let $\Lambda:= \operatorname{Hom}_{\mathrm{cont}} \left(\prod_{\ell\neq p}\mathbb Z_\ell^2,k^*\right)$ and $\Lambda^{\sigma_m}:=\{\chi\in\Lambda\mid \chi^{(p^{m})}=\chi\}$ where $m\in \mathbb{N}_{>0}$. So $\chi\in \Lambda^{\sigma_m}$ precisely when its values lie in $\mathbb F_{p^{m}}^*$. Then we have $|\Lambda^{\sigma_m}|=(p^{m}-1)^2$ by Lemma~\ref{Numberofcharacter}. 

Fix a positive integer $n$, define $\sigma_n:\Lambda\rightarrow \Lambda, \chi\mapsto \chi^{(p^n)}$. For $\chi\in\Lambda$, define the period of $\chi$ under $\sigma_n$ by $\peri_{\sigma_n}(\chi)$ to be the minimal positive integer $d$ such that $\chi=\chi^{(p^{nd})}$. For $\chi\in\Lambda$ of period $d$, the orbit of $\chi$ under $\sigma_n$ is $\orb_{{\sigma_n}}(\chi):=\{\chi,\chi^{(p^{n})},\cdots, \chi^{(p^{n(d-1)})}\}$ and we define $\OBT_{\sigma_n}(d):=\{\orb_{\sigma_n}(\chi)|\chi\in \Lambda,\,\peri_{\sigma_n}(\chi)=d\}$. Let $a_n(d)$ be the number of orbits of length $d$ under $\sigma_n$. Fix $e\in \mathbb{N}_{>0}$, we have 
$$\Lambda^{\sigma_n^e}=\bigcup_{\begin{smallmatrix}
    d|e,\chi\in \Lambda\\
    \peri_{\sigma_n}(\chi)=d
\end{smallmatrix}}\orb_{\sigma_n}(\chi),\quad \left|\Lambda^{\sigma_n^e}\right|=(p^{en}-1)^2,\quad\left|\bigcup\limits_{\begin{smallmatrix}
    \chi\in \Lambda\\
    \peri_{\sigma_n}(\chi)=d
\end{smallmatrix}}\orb_{\sigma_n}(\chi)\right|=da_n(d).$$
It follows that $(p^{en}-1)^2=\sum_{d\mid e}d\,a_n(d)$. By M\"obius inversion, this gives $a_n(d)= \frac{1}{d}\sum_{e\mid d}\mu\!\left(\frac{d}{e}\right)(p^{en}-1)^2$. 

For any nontrivial $k$-representation $(V,\rho)$ of $\pi_1^{\acute{e}t}(X,x)$, consider $V\simeq\bigoplus\limits_{i\in I} k_iV_{i}$ where $k_i\in\mathbb{N}_{>0}$, $V_i$ is indecomposable with $\dim_k V_i=m_i$ for each $i$ and $V_i\not\simeq V_j$ if $i\neq j$. Then for each $i$, there exists a pair $(\chi_i,[\rho_{m_i}])$ corresponding to the indecomposable $k$-representation $(V_i,\rho|_{V_i})$ of $\pi_1^{\acute{e}t}(X,x)$. Suppose $(V,\rho)$ satisfies $\rho^{(p^{n})}\cong \rho$ for some $n\in\mathbb{N}_{>0}$. Now we fix a $V_i$ and its corresponding pair $(\chi_i,[\rho_{m_i}])$. It follows that $\chi_i$ is of finite period under $\sigma_n$, otherwise $(\chi_i^{(p^{nj})},[\rho_{m_i}])$ will appear in this decomposition for infinitely many $j$, which contradicts $\dim_k V<\infty$. So for the same reason, if $\peri_{\sigma_n}(\chi_i)=d_i$, then we have 
$$V\simeq k_iV_{i}\bigoplus k_iV_i^{(p^n)}\bigoplus\cdots\bigoplus k_iV_i^{(p^{(d_i-1)n})}\bigoplus\left(\bigoplus_{i'} k_{i'}V_{i'}\right),$$
such that $V_i^{(p^{fn})}$ corresponds to $(\chi^{(p^{fn})},[\rho_{m_i}])$, $1\leq f\leq d_i-1$. Repeating this method, we obtain $V\simeq \bigoplus\limits_{t\in T} k_tW_t$, where $W_{t}=V_{i_t}\oplus V_{i_t}^{(p^{n})}\oplus\cdots\oplus V_{i_t}^{(p^{(d_{i_t}-1)n})}$ for some $V_{i_t}$ and $W_{t_1}\not\simeq W_{t_2}$ if $t_1\neq t_2$. 

In other words, for any $r$-dimensional $k$-representation $(V,\rho)$ of $\pi_1^{\acute{e}t}(X,x)$ satisfying $\rho^{(p^n)}\simeq \rho$, there exists a unique decomposition 
$$V\simeq \bigoplus_{t\in T} k_tW_t= \bigoplus_{m\in\mathbb{N}_{>0}}\bigoplus_{d\in\mathbb{N}_{>0}}\bigoplus_{\Omega\in\OBT_{\sigma_n}(d)}k_{(\Omega,[\rho_m])}W_{(\Omega,[\rho_m])},$$ 
such that $\rho_m$ is an $m$-dimensional indecomposable representation of $\mathbb{Z}_p$, $k_{(\Omega,[\rho_m])}\in \mathbb{N}$, $\Omega=\orb_{\sigma_{n}}(\chi)$ for some $\chi\in \Lambda$ and 
$$W_{(\Omega,[\rho_m])}=V_{(\chi,[\rho_m])}\oplus V_{(\chi^{(p^{n})},[\rho_m])}\oplus\cdots \oplus V_{(\chi^{(p^{(\peri_{\sigma_n}(\chi)-1)n})},[\rho_m])}$$ 
where $V_{(\chi^{(p^{fn})},[\rho_m])}$ corresponds to the pair $({\chi^{(p^{fn})},[\rho_m]})$ for $0\leq f\leq \peri_{\sigma_n}(\chi)-1$. Then $k_{(\Omega,[\rho_m])}\neq 0$ if and only if $(\chi,[\rho_m])$ corresponds to some nontrivial indecomposable $k$-subrepresentation of $(V,\rho)$. Since there are $a_n(d)$ orbits of length $d$ under $\sigma_n$ in $\Lambda$, the total generating function is
\[
\sum_{r\geq 0}N_r^{\mid n}T^r=\prod_{m\geq 1}\prod_{d\geq 1}(1+T^{md}+T^{2md}+\cdots)^{a_n(d)}=\prod_{m\geq 1}\prod_{d\geq 1}(1-T^{md})^{-a_n(d)}.
\]
Therefore, we have 
$$\prod_{m\geq 1}\prod_{d\geq 1}(1-T^{md})^{-a_n(d)}=\prod_{s\geq 1}(1-T^s)^{-\sum_{d\mid s}a_n(d)}.$$
By the binomial expansion, $(1-T^s)^{-\sum_{d\mid s}a_n(d)}=\sum\limits_{z\geq 0}\binom{z-1+\sum_{d\mid s}a_n(d)}{z}T^{sz}$. Hence we have
$$\sum_{r\geq 0}N_r^{\mid n}T^r=\prod_{s\geq 1}(1-T^s)^{-\sum_{d\mid s}a_n(d)}=\prod_{s\geq 1}\sum_{z\geq 0}\binom{z-1+\sum_{d\mid s}a_n(d)}{z}T^{sz}.$$
Therefore, the coefficient of $T^r$ is 
$$N_r^{\mid n}=\sum_{\begin{smallmatrix}
    z_1,\dots,z_r\geq0\\
    z_1+2z_2+\cdots+rz_r=r
\end{smallmatrix} }\prod_{s=1}^r\binom{z_s-1+\sum_{d\mid s}a_n(d)}{z_s}.$$
Note that $N_r^{\mid n}=\sum\limits_{t\mid n}N_r^{=t}$. Applying M\"obius inversion gives $N_r^{=n}=\sum\limits_{t\mid n}\mu(\frac{n}{t})N_r^{\mid t}$, which is
$$N_{r}^{=n}=\sum_{t\mid n}\mu\!\left(\frac{n}{t}\right)\sum_{\substack{z_1,\dots,z_r\geq 0\\z_1+2z_2+\cdots+rz_r=r}}\prod_{s=1}^r\binom{z_s-1+\sum_{d\mid s}\frac{1}{d}\sum_{e\mid d}\mu\!\left(\frac{d}{e}\right)(p^{te}-1)^2}{z_s}.$$
\end{proof}

\begin{Proposition}\label{supersingular}
Let $k$ be an algebraically closed field of characteristic $p>0$, $X$ a supersingular elliptic curve over $k$, $x\in X(k)$, $n\in\mathbb{N}_{>0}$, $N_r^{\mid n}(\text{resp. }N_r^{=n})$ denote the number of isomorphism classes of rank $r$ vector bundles $E$ on $X$ satisfying $(F_X^n)^*E\cong E($resp. whose minimal Frobenius period is exactly $n)$. Then we have the following generating series
\[
    \sum_{r\geq 0} N_r^{\mid n}T^r
    =
    \prod_{d\geq 1}
    \left(1-T^{d}\right)^{-\frac{1}{d}\sum_{e\mid d}\mu\!\left(\frac{d}{e}\right)(p^{en}-1)^2}.
\]
Here $\mu$ is the classical M\"obius function. Consequently, we have 
$$N_r^{\mid n}=\sum_{\begin{smallmatrix}
    z_1,\dots,z_r\geq0\\
    z_1+2z_2+\cdots+rz_r=r
\end{smallmatrix} }\prod_{d=1}^r\binom{z_d-1+\frac{1}{d}\sum_{e\mid d}\mu\!\left(\frac{d}{e}\right)(p^{en}-1)^2}{z_d},$$
$$N_{r}^{=n}=\sum_{t\mid n}\mu\!\left(\frac{n}{t}\right)\sum_{\substack{z_1,\dots,z_r\geq 0\\z_1+2z_2+\cdots+rz_r=r}}\prod_{d=1}^r\binom{z_d-1+\frac{1}{d}\sum_{e\mid d}\mu\!\left(\frac{d}{e}\right)(p^{te}-1)^2}{z_d}.$$
\end{Proposition}

\begin{proof}
Since $X$ is a supersingular elliptic curve, its
étale fundamental group is
\[
\pi_1^{\acute{e}t}(X,x)\simeq
\prod_{\ell\neq p}\mathbb Z_\ell^2.
\]
Then for any nontrivial $k$-representation $(V,\rho)$ of $\pi_1^{\acute{e}t}(X,x)$, since $\Ima(\rho)$ is a finite abelian group with $p\nmid |\Ima(\rho)|$, we have that $V$ has a direct sum decomposition of subrepresentations of dimension $1$. Therefore, there exists a decomposition $V=\oplus_{\chi\in \Lambda}k_\chi V_{\chi}$, where $\Lambda:=\Hom_{\mathrm{cont}}(\prod_{\ell\neq p}\mathbb Z_\ell^2,k^*)$, $k_\chi\in\mathbb{N}$, $(V_\chi,\rho|_{V_\chi})$ satisfies $\rho(a) v=\chi(a)v$ for any $a\in \prod_{\ell\neq p}\mathbb Z_\ell^2$ and for any $v\in V_\chi$, and $V_\chi$ is a $1$-dimensional representation if $V_\chi\neq0$.
 
Fix a positive integer $n$, define $\sigma_n:\Lambda\rightarrow \Lambda, \chi\mapsto \chi^{(p^n)}$. For $\chi\in\Lambda$, define the period of $\chi$ under $\sigma_n$ by $\peri_{\sigma_n}(\chi)$ to be the minimal positive integer $d$ such that $\chi=\chi^{(p^{nd})}$. For $\chi\in\Lambda$ of period $d$, the orbit $\orb_{{\sigma_n}}(\chi)$ of $\chi$ under $\sigma_n$ is $\{\chi,\chi^{(p^{n})},\cdots, \chi^{(p^{n(d-1)})}\}$ and we define $\OBT_{\sigma_n}(d):=\{\orb_{\sigma_n}(\chi)|\chi\in \Lambda,\,\peri_{\sigma_n}(\chi)=d\}$. Let $a_n(d)$ be the number of orbits of length exactly $d$ under $\sigma_n$. 

Suppose $(V,\rho)$ satisfies that $\rho^{(p^n)}\simeq \rho$. Similar to the proof of Proposition~\ref{ordinary}, we have the following decomposition
$$V= \bigoplus_{d\in\mathbb{N}_{>0}}\bigoplus_{\Omega\in\OBT_{\sigma_n}(d)}k_\Omega W_\Omega$$
such that $k_\Omega\in \mathbb{N}$, $\Omega=\orb_{\sigma_{n}}(\chi)$ for some $\chi\in \Lambda$ and 
$$W_\Omega=V_\chi\oplus V_{\chi^{(p^{n})}}\oplus\cdots \oplus V_{\chi^{(p^{(\peri_{\sigma_n}(\chi)-1)n})}}.$$
Then $k_{\Omega}\neq 0$ if and only if $V_\chi$ corresponds to some nontrivial indecomposable $k$-subrepresentation of $(V,\rho)$. Since there are $a_n(d)$ orbits of length $d$ under $\sigma_n$ in $\Lambda$, the total generating function is
\[
\sum_{r\geq 0}N_r^{\mid n}T^r=\prod_{d\geq 1}(1+T^{d}+T^{2d}+\cdots)^{a_n(d)}=\prod_{d\geq 1}(1-T^{d})^{-a_n(d)}.
\]
By the binomial expansion, $(1-T^d)^{-a_n(d)}=\sum\limits_{z\geq 0}\binom{z-1+a_n(d)}{z}T^{dz}$. Hence we have
$$\prod_{d\geq 1}(1-T^d)^{-a_n(d)}=\prod_{d\geq 1}\sum_{z\geq 0}\binom{z-1+a_n(d)}{z}T^{dz}.$$
Therefore, the coefficient of $T^r$ is 
$$N_r^{\mid n}=\sum_{\begin{smallmatrix}
    z_1,\dots,z_r\geq0\\
    z_1+2z_2+\cdots+rz_r=r
\end{smallmatrix} }\prod_{d=1}^r\binom{z_d-1+a_n(d)}{z_d}.$$
Note that $N_r^{\mid n}=\sum\limits_{t\mid n}N_r^{=t}$. Applying M\"obius inversion gives $N_r^{=n}=\sum\limits_{t\mid n}\mu(\frac{n}{t})N_r^{\mid t}$, which is
$$N_{r}^{=n}=\sum_{t\mid n}\mu\!\left(\frac{n}{t}\right)\sum_{\substack{z_1,\dots,z_r\geq 0\\z_1+2z_2+\cdots+rz_r=r}}\prod_{d=1}^r\binom{z_d-1+\frac{1}{d}\sum_{e\mid d}\mu\!\left(\frac{d}{e}\right)(p^{te}-1)^2}{z_d}.$$
\end{proof}

\begin{Corollary}\label{Cor}
    Let $k$ be an algebraically closed field of characteristic $p>0$, $X$ an elliptic curve over $k$, $n\in\mathbb{N}_{>0}$. Then we have  
    $$|\mathrm{Pic}^0(X)[p^n-1]|={(p^{n}-1)^2}.$$
\end{Corollary}

\begin{proof}
    Let $F_X:X\rightarrow X$ be the absolute Frobenius morphism, $L$ a line bundle on $X$. Then 
    $$F_X^{n*} L\cong L\Leftrightarrow L^{\otimes p^n}\cong L \Leftrightarrow L\in \mathrm{Pic}^0(X)[p^n-1].$$
    
    If $X$ is ordinary, then by Proposition~\ref{ordinary},
    $$|\mathrm{Pic}^0(X)[p^n-1]| =N_{1}^{|n}={(p^{n}-1)^2}.$$
    
    If $X$ is supersingular, then by Proposition~\ref{supersingular},
    $$|\mathrm{Pic}^0(X)[p^n-1]| =N_{1}^{|n}={(p^{n}-1)^2}.$$
    
\end{proof}


\end{document}